\documentclass{amsart}

\usepackage{amsthm,amssymb,latexsym,amsmath}
\usepackage{mathrsfs}
\usepackage{graphicx}

\newtheorem{theorem}{Theorem}[section]

\newtheorem{corollary}[theorem]{Corollary}

\newtheorem{definition}[theorem]{Definition}

\def\F{{\mathcal{F}}}

\newcommand{\sing}{\text{\rm Sing}}
\newcommand{\tang}{\text{\rm Tang}}

\begin{document}
\title[Residues for maps generically transverse to distributions]{Residues
for maps generically transverse to distributions}
\date{\today }
\author{Leonardo M. C\^amara}
\author{Maur\'icio Corr\^ea }
\address{\emph{Leonardo M. C\^amara}: Depto. de Mat. -- CCE, Universidade
Federal do Esp\'irito Santo -- UFES}
\curraddr{Av. Fernando Ferrari 514, 29075-910, Vit\'oria - ES, Brasil.}
\email{leonardo.camara@ufes.br}
\address{\emph{M. Corr\^ea}: Depto. de Mat.--ICEX, Universidade Federal de
Minas Gerais--UFMG}
\curraddr{Av. Ant\^onio Carlos 6627, 31270-901, Belo Horizonte-MG, Brasil.}
\email{mauriciomatufmg@gmail.com}

\begin{abstract}
We show a residues formula for maps generically transversal to regular
holomorphic distributions.  
\end{abstract}

\maketitle

\section{Introduction}
Let $f:X\longrightarrow Y$ be a   singular  holomorphic
map between  complex   manifolds $X$ and $Y$, with $\dim(X):=n\geq m=:\dim(Y)$, having   generic fiber $F$. 
Consider the singular set of $f$ defined by
$$
S:=\sing(f)=\{p\in X:  \mbox{rank}(df(p))<m\ \}.
$$
If $Y=C$ is a   curve,     Iversen in \cite{Iversen}  proved the 
following multiplicity formula
\begin{equation*}
\chi(X)-\chi(F)\cdot\chi(C)=(-1)^{n}\sum_{p\in  \sing(f)}\mu _{p}(f),
\end{equation*}
where   $
\mu_{p}(f)$ is the Milnor number of $f$ at $p$.   Izawa and Suwa \cite{Izawa-Suwa2003} generalized Iversen's result
for the case where $X$ is possibly a singular variety.

A   generalization of the multiplicity formula for maps was  given by Diop in \cite{Diop}. In his work he generalized some
formulas involving the Chern classes given previously by  Iversen  \cite{Iversen}, Brasselet \cite{Bra1, Bra2}, and
 Schwartz \cite{Schwartz}. More precisely,  Diop showed that 
if  $S$ is smooth and  $\dim(S)=m-1$ then 
\begin{equation*}
\chi(X)-\chi(F)\chi(Y)=(-1)^{n-m+1}\sum_{j}\mu_{j}\int_{S_{j}}c_{q-1}[%
\left. (f^{\ast}TY)\right\vert _{S_{j}}-\mathcal{L}_{j}],
\end{equation*}
where $S=\cup S_{j}$ is the decomposition of $S$ into irreducible
components, $\mu_{j}=\mu(\left. f\right\vert \Sigma_{j})$  is the
Milnor number of the restriction of $f$ to a transversal section $\Sigma_{j}$
to $S_{j}$ at a regular point \thinspace$p_{j}\in S_{j}$, and $\mathcal{L}%
_{j}$ is the line bundle over $S_{j}$ given by the decomposition $%
f^{\ast}df(\left. TX\right\vert _{S_{j}})\oplus\mathcal{L}_{j}=\left.
f^{\ast}(TY)\right\vert _{S_{j}}$.

On the other hand, Brunella in \cite{Bru} introduced the notion of tangency
index of a germ of curve with respect to a germ of holomorphic foliation:
given a reduced curve $C$ and a foliation $\mathcal{F}$ (possibly singular)
on a complex compact surface. Suppose that $C$ is not invariant by $\mathcal{%
F}$ and that $C$ and $\mathcal{F}$ are given locally by $\{f=0\}$ and a
vector field $v$, respectively. The tangency index $I_{p}(\mathcal{F},C)$ of 
$C$ with respect to $\mathcal{F}$ at $p$ is given by the intersection number 
\begin{equation*}
I_{p}(\mathcal{F},C)=\dim_{\mathbb{C}}\mathcal{O}_{2}/(f,v(f)).
\end{equation*}
Using this index, Brunella proved the following formula 
\begin{equation*}
c_{1}(\mathcal{O}(C))^{2}-c_{1}(T_{\mathcal{F}})\cap c_{1}(\mathcal{O}%
(C))=\sum_{p\in\tang(\mathcal{F},C)}I_{p}(\mathcal{F},C),
\end{equation*}
where $T_{\mathcal{F}}$ is the tangent bundle of $\mathcal{F}$ and $\tang(\mathcal{F},C)$ denotes the non-transversality loci of $C$ with
respect to $\mathcal{F}$. In \cite{Honda} and \cite{Honda2}, T. Honda also
studied Brunella's tangency formula. Distributions and foliations transverse to certain domains in $\mathbb{C}^n$ has been studied by  Bracci and Sc\'ardua  in  \cite{BS}  and Ito and Sc\'ardua  in \cite{IS}.

Recently, Izawa \cite{Izawa} generalized certain results due to Diop \cite
{Diop} in the foliated context. More precisely, let $f:X\longrightarrow (Y,%
\mathcal{F})$ be a holomorphic map such that $\mathcal{F}$ is a regular
holomorphic foliation of codimension one in $Y$. 
Let $S(f,\F)$ be  the set of points where $f$ fails to be transverse to $\F$. 
Suppose  $S(f,\F)$ is given by isolated points and let $\widetilde{\mathcal{F}}:=f^{\ast}%
\mathcal{F}$. Since $\mathcal{F}$ is regular, we may find local coordinates
in a neighborhood of $p\in \sing(f)$ and $f(p)$ in such a way that $%
f=(f_{1},\cdots ,f_{m})$ and $\widetilde{\mathcal{F}}$ is given by $%
\ker(df_{m})$ nearby $p$. If we pick $g_{i}:=\frac{\partial f_{m}}{\partial
x_{i}}$ (i.e., $df_{m}=g_{1}dx_{1}+\cdots+g_{n}dx_{n}$), then%
\begin{equation*}
\mathcal{\chi}(X)-\sum_{i=1}^{r}f_{\ast}(c_{n-i}(T_{X})\cap\lbrack X])\cap
c_{1}(\mathcal{N}_{\mathcal{F}})^{i}=(-1)^{n}\sum_{p\in S(f,\F)}%
\mbox{Res}_{p}\left[ 
\begin{array}{c}
dg_{1}\wedge\cdots\wedge dg_{m} \\ 
g_{1},\cdots,g_{m}%
\end{array}
\right] ,
\end{equation*}
where $\mathcal{N}_{\mathcal{F}}$ denotes the normal sheaf of $\mathcal{F}.$

In this paper we generalize the above results for a regular distribution $%
\mathcal{F}$ in $Y$ of any codimension with the following residual formula
for the non-transversality points of $f(X)$ with respect to $\mathcal{F}$.

In order to state our main result, let us introduce some notions. 
Let $f:X\longrightarrow(Y,{\mathcal{F}})$ be a holomorphic map  and suppose  that $X$ and $Y$ are projective manifolds. 
We say that the set of points in $X$ where $f$ fails to be transversal to $\mathcal{F}$
is the \emph{ramification locus} of $f$ with respect to $\mathcal{F}$, and
denote it by $S(f,\mathcal{F})$.
The set $R(f,\mathcal{F}):=f(S(f,\mathcal{F}))$ is
called the \emph{\ branch locus} or the set of \emph{\ branch points} of $f$ with respect to $\mathcal{F}$.
Let $S(f,\mathcal{F})=\cup S_{j}$ be the decomposition of $S$ into irreducible components,
then we denote by $\mu(f,\mathcal{F},S_{j})$ the multiplicity of $S_{j}$ and
call it the \emph{ramification multiplicity of } $f$\emph{\ along }$S_{j}$ with
respect to $\mathcal{F}$. As usual, we denote by $[W]$ the  class in the Chow group of $X$  of the  subvariety  $W\subset X$. 
The class  $f_{\ast}[S_{j}]=:[R_{j}]$ is called a 
\emph{branch class} of $f$.
Observe that $R(f,\mathcal{F})$ is the set of tangency points between $f(X)$ and $\mathcal{F}$ if $%
\dim(X)\leq\dim(Y)$.

\begin{theorem}
\label{main} Let $f:X\longrightarrow(Y,{\mathcal{F}})$ be a holomorphic map 
of generic rank $r$ and $\mathcal{F}$ a non-singular distribution of
codimension $k$ on $Y$. Suppose the ramification locus of $f$ with respect
to $\mathcal{F}$ has codimension $n-k+1$, then 
\begin{equation*}
f_{\ast}(c_{n-k+1}(T_{X})\cap\lbrack X])+\sum_{i=1}^{r}(-1)^{i}f_{\ast
}(c_{n-k+1-i}(T_{X})\cap\lbrack X])\cap s_{i}(\mathcal{N}_{\mathcal{F}%
}^{\ast })=(-1)^{n-k+1}\sum_{R_{j}\subset R}\mu(f,\mathcal{F},S_{j})[R_{j}],
\end{equation*}
where $s_{i}(\mathcal{N}_{\mathcal{F}}^{\ast})$ is the $i$-th Segre class of 
$\mathcal{N}_{\mathcal{F}}^{\ast}$.
\end{theorem}

Some consequences of this result are the following.

\begin{corollary}[Izawa]
\label{Izawa}If $k=1$, then 
\begin{equation*}
\mathcal{\chi}(X)-\sum_{i=1}^{r}f_{\ast}(c_{n-i}(T_{X})\cap\lbrack X])\cap
c_{1}(\mathcal{N}_{\mathcal{F}})^{i}=(-1)^{n}\sum_{p\in S(f,\F)} \text{\rm Res}_p  \left[ 
\begin{array}{c}
dg_{1}\wedge\cdots\wedge dg_{m} \\ 
g_{1},\cdots,g_{m}%
\end{array}
\right].
\end{equation*}
\end{corollary}

In fact, if $k=1$ we have $c_{n}(T_{X})\cap\lbrack X]=\mathcal{\chi}(X)$ by
the Chern-Gauss-Bonnet Theorem. Since $\mathcal{N}_{\mathcal{F}}^{\ast}$ is
a line bundle, then $s_{i}(\mathcal{N}_{\mathcal{F}}^{\ast})=(-1)^{i}c_{1}(%
\mathcal{N}_{\mathcal{F}}^{\ast})^{i}$ for all $i$. The above Izawa's
formula \cite[Theorem 4.1]{Izawa} implies the multiplicity formula 
\begin{equation*}
\chi(X)-\chi(F)\cdot\chi(C)=(-1)^{n}\sum_{p\in\sing(f)}\mu _{p}(f).
\end{equation*}

\begin{corollary}[Tangency formulae]
\label{Tangency} Let $X \subset Y$ be a $k$-dimensional submanifold
generically transverse to a non-singular distribution $\mathcal{F}$ on $Y$
of codimension $k$. Then 
\begin{equation*}
[c_{1}(N_{X|Y})-c_{1}(T_{\mathcal{F}})]\cap\lbrack X]=\sum_{R_{j}\subset
R}\mu(f,\mathcal{F},S_{j})[R_{j}].
\end{equation*}
In particular, if $\det(T_{\mathcal{F} })|_{X}-\det(N_{X|Y})$ is ample, then 
$X$ is tangent to ${\mathcal{F}}$.
\end{corollary}

If $X=C$ is a curve on a surface $Y$, we have $[C]=c_{1}(\mathcal{O}%
(C))=c_{1}(N_{X|Y})$. This yields Brunella's formula 
\begin{equation*}
c_{1}(\mathcal{O}(C))^{2}-c_{1}(T_{\mathcal{F}})\cap c_{1}(\mathcal{O}%
(C))=\sum_{p\in \tang(\mathcal{F},C)}I_{p}(\mathcal{F},C).
\end{equation*}%
Moreover, this formula coincides with the Honda's formula \cite{Honda2} in
case $\mathcal{F}$ is a one-dimensional foliation and $X$ is a curve.

In Section 3, we prove Theorem \ref{main} and Corollary \ref{Tangency}.
\\
\\
\noindent\textbf{Acknowlegments.} The second named author was partially
supported by CAPES, CNPq and Fapesp-2015/20841-5 Research Fellowships.
 Finally, we would like to thank the referee by the suggestions, comments and
improvements to the exposition.

\section{Holomorphic distributions}

Let $X$ be a complex manifold of dimension $n$.

\begin{definition}\label{defifol}
A codimension $k$ distribution $\mathcal{F}$ on $X$ is given by an exact
sequence 
\begin{equation}  \label{defi_fol}
\mathcal{F}:\ 0\longrightarrow\mathcal{N}_{\mathcal{F}}^{\ast
}\longrightarrow\Omega_{X}^{1}\longrightarrow\Omega_{\mathcal{F}}\longrightarrow0,
\end{equation}
where $\mathcal{N}_{\mathcal{F}}^{\ast}$ is a coherent sheaf of rank $%
k\leq\dim(X)-1$ and $\Omega_{\mathcal{F}}$ is a torsion free sheaf. We say
that $\mathcal{F}$ is a foliation if at the level of local sections we have $%
d(\mathcal{N}_{\mathcal{F}}^{\ast})\subset\mathcal{N}_{\mathcal{F}%
}^{\ast}\wedge\Omega _{X}^{1}$. The singular set of the distribution $%
\mathcal{F}$ is defined by $\sing(\mathcal{F}):=\sing%
(\Omega_{\mathcal{F}})$. We say that $\mathcal{F}$ is regular if $\sing(\mathcal{F})=\emptyset$.
\end{definition}
Taking determinants of the map $\mathcal{N}_{\mathcal{F}}^{\ast}%
\longrightarrow\Omega_{X}^{1}$, we obtain a map: 
\begin{equation*}
\det(\mathcal{N}_{\mathcal{F}}^{\ast})\longrightarrow\Omega_{X}^{k},
\end{equation*}
which induces a twisted holomorphic $k$-form $\omega\in
H^{0}(X,\Omega_{X}^{k}\otimes\det(\mathcal{N}_{\mathcal{F}}^{\ast})^{\ast})$.
 Therefore, a distribution can be induced by a twisted holomorphic $k$-form 
$H^{0}(X,\Omega_{X}^{k}\otimes\det(\mathcal{N}_{\mathcal{F}}^{\ast})^{\ast})$
which is locally decomposable outside the singular set of ${\mathcal{F}}$.
That is, for each point $p\in X\setminus\sing({\mathcal{F}})$ there
exists a neighborhood $U$ and holomorphic $1$-forms $\omega_{1},\dots
,\omega_{k}\in H^{0}(U,\Omega_{U}^{1})$ such that 
\begin{equation*}
\omega|_{U}=\omega_{1}\wedge\cdots\wedge\omega_{k}.
\end{equation*}
Moreover, if ${\mathcal{F}}$ is a foliation then by Definition \ref{defifol}
 we have 
\begin{equation*}
d\omega_{i}\wedge\omega_{1}\wedge\cdots\wedge\omega_{k}=0
\end{equation*}
for all $i=1,\dots,k$.  
The tangent sheaf  of $\mathcal{F}$ is the  coherent sheaf of rank $(n-k)$  given by  
\begin{equation*}
T_{\mathcal{F}}=\{v\in T_{X};\ i_{v}\omega=0\}.
\end{equation*}
The normal sheaf of $\mathcal{F}$ is defined by $\mathcal{N}_{\mathcal{F} }=T_X/T_{\mathcal{F}}.$
It is worth noting that $\mathcal{N}_{\mathcal{F} }\neq  (\mathcal{N}_{\mathcal{F}}^{\ast} )^{\ast}$ whenever $\sing(\mathcal{F})\neq\emptyset$. Dualizing the sequence (\ref{defi_fol}) one obtains the exact sequence
\begin{equation*}
\ 0\longrightarrow T_{\mathcal{F}}\longrightarrow
T_{X}\longrightarrow (\mathcal{N}_{\mathcal{F}}^{\ast} )^{\ast}\longrightarrow \mathcal{E}xt^1(\Omega_{\mathcal{F}}, \mathcal{O}_X) \longrightarrow 0,
\end{equation*}
so that there is an exact sequence
\begin{equation*}
\ 0\longrightarrow  \mathcal{N}_{\mathcal{F}}   \longrightarrow (\mathcal{N}_{\mathcal{F}}^{\ast} )^{\ast}\longrightarrow \mathcal{E}xt^1(\Omega_{\mathcal{F}}, \mathcal{O}_X) \longrightarrow 0. 
\end{equation*}

\begin{definition}
Let $V\subset X$ an analytic subset. We say that $V$ is tangent to ${%
\mathcal{F}}$ if $T_{p}V\subset(T_{{\mathcal{F}}})_{p}$, for all $p\in
V\setminus \sing(V).$
\end{definition}

\section{Proof of the main results}

We begin by proving the main theorem.

\begin{proof}[Proof of Theorem \protect\ref{main}]
Consider a map $f:X\longrightarrow Y$ and let $(U,x)$ and $(V,y)$ be local
systems of coordinates for $X$ and $Y$ such that $f(U)\subset V$. Since $%
\mathcal{F}$ is a regular distribution, we may suppose that it is induced on 
$U$ by the $k$-form $\omega_{1}\wedge\cdots\wedge \omega_{k}$. Therefore,
the ramification locus of $f$ with respect to $\mathcal{F}$ on $U$ is given
by 
\begin{equation*}
S(f,\mathcal{F})|_{U}=\{f^{\ast}(\omega_{1}\wedge\cdots\wedge\omega_{k})=f^{\ast}(\omega
_{1})\wedge\cdots\wedge f^{\ast}(\omega_{k})=0\}.
\end{equation*}
In other words, the ramification locus $S(f,\mathcal{F})$ coincides with $\sing%
(f^{\ast}(\mathcal{F}))$. 

Let us denote $\widetilde {\mathcal{F}}:=f^{\ast}(%
\mathcal{F})$. Let $\{U_{\alpha}\}$ be a covering of $Y$ such that the
distribution ${\mathcal{F}}$ is induced on $U_{\alpha}$ by the holomorphic $%
1 $-forms $\omega_{1}^{\alpha},\dots,\omega_{k}^{\alpha}$. Hence, on $%
U_{\alpha}\cap U_{\beta}\neq\emptyset$ we have $(\omega_{1}^{\alpha}\wedge%
\cdots\wedge\omega_{k}^{\alpha})=g_{\alpha\beta}(\omega
_{1}^{\beta}\wedge\cdots\wedge\omega_{k}^{\beta}),$ where $\{g_{\alpha\beta
}\}$ is a cocycle generating the line bundle $\det(\mathcal{N}_{\mathcal{F}%
}^{\ast})^{\ast}.$ Then the distribution $\widetilde{\mathcal{F}}$ is
induced locally by $f^{\ast}(\omega_{1}^{\alpha}),\dots,f^{\ast}(%
\omega_{k}^{\alpha})$. This shows that $\mathcal{N}_{{\widetilde{\mathcal{F}}%
}}^{\ast}$ is locally free. Therefore the singular set of $\widetilde{\mathcal{F}%
}$ is the loci of degeneracy of the induced map 
\begin{equation*}
\mathcal{N}_{{\widetilde{\mathcal{F}}}}^{\ast}\longrightarrow\Omega_{X}^{1}
\end{equation*}
By hypothesis, the ramification locus of $f$ with respect to $\mathcal{F}$,
which is given by $\sing(\widetilde{\mathcal{F}})$, has codimension 
$n-k+1$, then it follows from Thom-Porteous formula \cite{Fulton} that 
\begin{equation*}
c_{n-k+1}(\Omega_{X}^{1}-\mathcal{N}_{{\widetilde{\mathcal{F}}}%
}^{\ast})\cap [X]=\sum_{j}\mu_{j}[S_{j}], 
\end{equation*}
where $\mu_{j}$ is the multiplicity of the irreducible component $S_{i}$. It
follows from $c(\Omega_{X}^{1}-\mathcal{N}_{{\widetilde{\mathcal{F}}}}^{\ast
})=c(\Omega_{X}^{1})\cdot s(\mathcal{N}_{{\widetilde{\mathcal{F}}}}^{\ast})$
that 
\begin{equation*}
c_{n-k+1}(\Omega_{X}^{1}-\mathcal{N}_{{\widetilde{\mathcal{F}}}%
}^{\ast})=\sum_{i=0}^{n-k+1}c_{n-k+1-i}(\Omega_{X}^{1})\cap s_{i}(\mathcal{N}%
_{{\widetilde{\mathcal{F}}}}^{\ast}),
\end{equation*}
where $s_{i}(\mathcal{N}_{{\widetilde{\mathcal{F}}}}^{\ast})$ is the $i$-th
Segre classe of $\mathcal{N}_{{\widetilde{\mathcal{F}}}}^{\ast}$. Since $%
X_{0}:=X-\sing(\widetilde{\mathcal{F}})$ is a dense and open subset
of $X$, then by taking the cap product we have 
\begin{align*}
c_{n-k+1}(\Omega_{X}^{1}-\mathcal{N}_{{\widetilde{\mathcal{F}}}%
}^{\ast})\cap\lbrack X] & =c_{n-k+1}(\Omega_{X}^{1}-\mathcal{N}_{{%
\widetilde {\mathcal{F}}}}^{\ast})\cap\lbrack X_{0}] \\
& =\sum_{i=0}^{n-k+1}(c_{n-k+1-i}(\Omega_{X}^{1}))\cap\lbrack X_{0}]\cap
s_{i}(f^{\ast}\mathcal{N}_{\mathcal{F}}^{\ast}).
\end{align*}
It follows from the projection formula that%
\begin{equation*}
f_{\ast}(c_{n-k+1}(\Omega_{X}^{1}-\mathcal{N}_{{\widetilde{\mathcal{F}}}%
}^{\ast}))\cap\lbrack X])=\sum_{i=0}^{n-k+1}f_{\ast}(c_{n-k+1-i}(\Omega
_{X}^{1})\cap\lbrack X])\cap s_{i}(\mathcal{N}_{\mathcal{F}}^{\ast})=\sum
_{j}\mu_{j}f_{\ast}[S_{j}].
\end{equation*}
\end{proof}

\bigskip

Now, we prove our tangency formulae as a consequence of the main theorem.

\begin{proof}[Proof of Corollary \protect\ref{Tangency}]
Let $i:X\hookrightarrow Y$ be the inclusion map. It follows from Theorem \ref%
{main} that 
\begin{equation*}
i_{\ast}(c_{1}(T_{X})\cap\lbrack X])-i_{\ast}([X])\cap s_{1}(\mathcal{N}_{%
\mathcal{F}}^{\ast})=-\sum_{R_{j}\subset R}\mu(f,\mathcal{F},S_{j})[R_{j}].
\end{equation*}
On the one hand, we have $c_{1}(T_{Y}|_{X})=c_{1}(N_{X|Y})+c_{1}(T_{X})$,
and on the other hand, we have $c_{1}(T_{Y}|_{X})=c_{1}(T_{\mathcal{F}%
}|_{X})+c_{1}(\mathcal{N}_{\mathcal{F}}|_{X})$. Since $s_{1}(\mathcal{N}_{%
\mathcal{F}}^{\ast})=-c_{1}(\mathcal{N}_{\mathcal{F}}^{\ast})=c_{1}(\mathcal{%
N}_{\mathcal{F}})$, we obtain 
\begin{equation*}
\lbrack c_{1}(N_{X|Y})-c_{1}(T_{\mathcal{F}})]\cap\lbrack X]=\sum
_{R_{j}\subset R}\mu(f,\mathcal{F},S_{j})[R_{j}].
\end{equation*}
Now notice that, by construction, the cycle 
\begin{equation*}
Z=\sum_{R_{j}\subset R}\mu(f,\mathcal{F},S_{j})[R_{j}]
\end{equation*}
is an effective divisor on $X$, since $\mu(f,\mathcal{F},S_{j})\geq0$. If
the line bundle $\det(T_{\mathcal{F}}))|_{X}-\det(N_{X|Y})=-[\det(N_{X|Y})-%
\det(T_{\mathcal{F}})|_{X}]$ is ample, we obtain 
\begin{equation*}
0<-[\det(N_{X|Y})-\det(T_{\mathcal{F}}))|_{X}]\cdot C=-Z\cdot C,
\end{equation*}
for all irreducible curve $C\subset X$. If $X$ is not invariant by ${%
\mathcal{F}}$ and $\det(T_{\mathcal{F}})|_{X}-\det(N_{X|Y})$ is ample, we
obtain an absurd. In fact, in this case $Z\cdot C<0$, contradicting the fact
that $Z$ is effective.\bigskip
\end{proof}

\section{Examples}

\subsection{Integrable example}

This example is inspired by an example due to Izawa \cite{Izawa}.

\noindent Consider $Y=\mathbb{P}^{3}\times\mathbb{P}^{1}\times\mathbb{P}^{1}$
and the subvariety $X=F^{-1}(0)\cap g^{-1}(0)$ given by the homogenous
equations%
\begin{equation*}
F(x,y,z)=\sum_{i=0}^{3}x_{i}^{\ell},\ \ G(x,y,z)=\sum_{i=0}^{1}x_{i}y_{i},
\end{equation*}
where $([x],[y],[z])=((x_{0}:x_{1}:x_{2}:x_{3}),(y_{0}:y_{1}),(z_{0}:z_{1}))%
\in Y$ are homogeneous coordinates. By a straightforward calculation one may
verify that $X$ is smooth. In $Y$ we consider the foliation $\mathcal{F}$
given by the fibers of the map $\pi:$ $\mathbb{P}^{3}\times\mathbb{P}%
^{1}\times\mathbb{P}^{1}\longrightarrow\mathbb{P}^{1}\times\mathbb{P}^{1}$
and let $f:X\rightarrow Y$ be the inclusion map. We will analyze the branch
points of the $f$ with respect to ${\mathcal{F}}$. \newline
\noindent A simple but exhaustive calculation shows that there is no branch
point in the hypersurface $x_{0}=0$, thus we concentrate in the Zariski
open set $x_{0}\neq0$.

\paragraph{The affine charts for $y_{0}\neq0$.}

In the affine charts for $x_{0}\neq0$ and $y_{0}\neq0$ the equations
defining $X$ assume the form%
\begin{align*}
1+x^{\ell}+y^{\ell}+z^{\ell} & =0, \\
1+ux & =0,
\end{align*}
where $(1:x:y:z)=(1:\frac{x_{1}}{x_{0}}:\frac{x_{2}}{x_{0}}:\frac{x_{3}}{%
x_{0}})$ and $(1:v)=(1:\frac{y_{1}}{x_{0}})$. This yields the
parametrization of $X$ given by%
\begin{align*}
x & =(-1)^{\frac{1}{\ell}}(y^{\ell}+z^{\ell}+1)^{\frac{1}{\ell}} \\
v & =(-1)^{\frac{1+\ell}{\ell}}(y^{\ell}+z^{\ell}+1)^{-\frac{1}{\ell}}.
\end{align*}
Now, recall that the leaves of $\mathcal{F}$ are given by $\{const\}\times\mathbb{C}$, hence the tangency points between $X$ and $\mathcal{F}$
are the solutions to the equation $du=u_{y}dy+u_{z}dz=0$. Thus the set of
tangency points coincides with the solutions of the following system of
equations 
\begin{align*}
0 & =\frac{\partial v}{\partial y}=(-1)^{-\frac{1}{\ell}}y^{\ell-1}(y^{\ell
}+z^{\ell}+1)^{-\frac{\ell+1}{\ell}} \\
0 & =\frac{\partial v}{\partial z}=(-1)^{-\frac{1}{\ell}}z^{\ell-1}(y^{\ell
}+z^{\ell}+1)^{-\frac{\ell+1}{\ell}}
\end{align*}
or in other words%
\begin{equation}
\left\{ 
\begin{array}{l}
x=(-1)^{\frac{1}{\ell}} \\ 
y^{\ell-1}=0 \\ 
z^{\ell-1}=0 \\ 
v=-(-1)^{-\frac{1}{\ell}}%
\end{array}
\right.  \label{sol-x0-y0}
\end{equation}
The solutions to this system of equations are given in terms of homogeneous
coordinates by%
\begin{equation*}
S_{k}^{0,0}=\{(1:\alpha_{k}:0:0)\}\times\{(1:-1/\alpha_{k})\}\times \mathbb{P%
}^{1},
\end{equation*}
where $\alpha_{k}=\exp(\frac{(2k+1)\pi i}{\ell})$, $k=0,\ldots,\ell-1$. Note
that $S_{k}^{0,0}$ is a solution with multiplicity $(\ell-1)^{2}$ and that
these solutions are contained in the codimension $2$ variety given by $%
x_{2}=x_{3}=0$.

\paragraph{The affine chart for $y_{1}\neq0$.}

On the other hand in the affine charts for $x_{0}\neq0$ and $y_{1}\neq0$ the
equations defining $X$ assume the form%
\begin{align*}
1+x^{\ell}+y^{\ell}+z^{\ell} & =0, \\
u+x & =0,
\end{align*}
where $(1:x:y:z)=(1:\frac{x_{1}}{x_{0}}:\frac{x_{2}}{x_{0}}:\frac{x_{3}}{%
x_{0}})$ and $(u:1)=(\frac{y_{0}}{y1}:1)$. This leads to the parametrization
of $X$ given by%
\begin{align*}
x & =(-1)^{\frac{1}{\ell}}(y^{\ell}+z^{\ell}+1)^{\frac{1}{\ell}} \\
u & =(-1)^{\frac{1+\ell}{\ell}}(y^{\ell}+z^{\ell}+1)^{\frac{1}{\ell}}.
\end{align*}
Since the leaves of $\mathcal{F}$ are given by $\{ const\}\times
\mathbb{C}$, then the tangency points between $X$ and $\mathcal{F}$ are the
solutions to the equation $du=u_{y}dy+u_{z}dz=0$. Therefore the set of
tangency points coincides with the solution to the system of equations%
\begin{align*}
0 & =\frac{\partial u}{\partial y}=(-1)^{\frac{\ell+1}{\ell}}y^{\ell
-1}(y^{\ell}+z^{\ell}+1)^{\frac{1-\ell}{\ell}} \\
0 & =\frac{\partial u}{\partial z}=(-1)^{\frac{\ell+1}{\ell}}z^{\ell
-1}(y^{\ell}+z^{\ell}+1)^{\frac{1-\ell}{\ell}}
\end{align*}
or in other words with the solutions to the system of equations%
\begin{equation}
\left\{ 
\begin{array}{l}
x=(-1)^{\frac{1}{\ell}} \\ 
y^{\ell-1}=0 \\ 
z^{\ell-1}=0 \\ 
u=-(-1)^{\frac{1}{\ell}}%
\end{array}
\right.  \label{sol-x0-y1}
\end{equation}
In homogeneous coordinates the solutions to this system of equations are
given by%
\begin{equation*}
S_{k}^{0,1}=\{(1:\alpha_{k}:0:0)\}\times\{(-\alpha_{k}:1)\}\times \mathbb{P}%
^{1},
\end{equation*}
where $\alpha_{k}=\exp(\frac{(2k+1)\pi i}{\ell})$, $k=0,\ldots,\ell-1$. Note
that $S_{k}^{0,1}$ is a solution with multiplicity $(\ell-1)^{2}$ and that
this solution is contained in the codimension $2$ variety $x_{2}=x_{3}=0$.
Notice also that $S_{k}^{0,1}=S_{k}^{0,0}$ for all $k=0,\ldots,\ell-1$.

\paragraph{The residual formula.}

Consider the projections $\pi_{1}:Y=\mathbb{P}^{3}\times\mathbb{P}^{1}\times%
\mathbb{P}^{1}\rightarrow\mathbb{P}^{3}$, $\pi_{2}:Y=\mathbb{P}^{3}\times%
\mathbb{P}^{1}\times\mathbb{P}^{1}\rightarrow\mathbb{P}^{1}$, $\pi_{3}:Y=%
\mathbb{P}^{3}\times\mathbb{P}^{1}\times\mathbb{P}^{1}\rightarrow\mathbb{P}%
^{1}$ and $\rho:Y=\mathbb{P}^{3}\times\mathbb{P}^{1}\times\mathbb{P}%
^{1}\rightarrow\mathbb{P}^{1}\times\mathbb{P}^{1}$. As usual, we denote a
line bundle on $Y$ by $\mathcal{O}(a,b,c):=\pi_{1}^{\ast }\mathcal{O}_{%
\mathbb{P}^{3}}(a)\otimes\pi_{2}^{\ast}\mathcal{O}_{\mathbb{P}%
^{1}}(b)\otimes\pi_{3}^{\ast}\mathcal{O}_{\mathbb{P}^{1}}(c)$, with $a,b,c\in%
\mathbb{Z}$. Now denote $h_{3}=c_{1}(\mathcal{O}(1,0,0))$, $h_{1,1}=c_{1}(%
\mathcal{O}(0,1,0))$, and $h_{1,2}=c_{1}(\mathcal{O}(0,0,1))$.

Summing up, the set of non-transversal points is given by the following
cycle 
\begin{equation*}
S=\sum_{k=0}^{\ell-1}(\ell-1)^{2}S_{k}^{0,0}.
\end{equation*}
Since $[S_{k}^{0,0}]=h_{3}^{3}\cdot h_{1,1}$, we concluded that 
\begin{align*}
\lbrack S] & =(\ell-1)^{2}\sum_{k=0}^{\ell-1}[S_{k}^{0,0}] \\
& =\ell(\ell-1)^{2}h_{3}^{3}\cdot h_{1,1}.
\end{align*}

Recall that $n=3$, $k=2$, and $r=2$, thus the left side of the formula
stated in Theorem \ref{main} assumes the form 
\begin{align*}
& f_{\ast}(c_{n-k+1}(T_{X})\cap\lbrack X])+\sum_{i=1}^{r}(-1)^{i}f_{\ast
}(c_{n-k+1-i}(T_{X})\cap\lbrack X])\cap s_{i}(\mathcal{N}_{\mathcal{F}%
}^{\ast })= \\
& =c_{2}(T_{X})\cap\lbrack X]-c_{1}(T_{X})\cap\lbrack X]\cap s_{1}(\mathcal{N%
}_{\mathcal{F}}^{\ast})+c_{0}(T_{X})\cap\lbrack X]\ \cap s_{2}(\mathcal{N}_{%
\mathcal{F}}^{\ast}) \\
& =\{\left. c_{2}(T_{X})-c_{1}(T_{X})\cap s_{1}(\mathcal{N}_{\mathcal{F}%
}^{\ast})+s_{2}(\mathcal{N}_{\mathcal{F}}^{\ast})\right\} \cap\lbrack X].
\end{align*}

Since the associated line bundles of $V(x_{0}^{\ell}+x_{1}^{\ell}+x_{2}^{%
\ell }+x_{3}^{\ell})$ and $V(x_{0}y_{0}+x_{1}y_{1})$ are $\mathcal{O}%
(\ell,0,0)$ and $\mathcal{O}(1,1,0)$, respectively, we have the short exact
sequence 
\begin{equation*}
0\longrightarrow T_{X}\longrightarrow T_{Y}|_{X}\longrightarrow\mathcal{O}%
(\ell,0,0)\oplus\mathcal{O}(1,1,0)|_{X}\longrightarrow0.
\end{equation*}
Now let $h_{3}=c_{1}(\mathcal{O}(1,0,0))$, $h_{1,1}=c_{1}(\mathcal{O}%
(0,1,0)) $, and $h_{1,2}=c_{1}(\mathcal{O}(0,0,1))$, then by the Euler
sequence for multiprojective spaces \cite{CS}, we conclude that 
\begin{equation*}
c(T_{Y})=(1+h_{3})^{4}(1+h_{1,1})^{2}(1+h_{1,2})^{2}.
\end{equation*}
with relations $(h_{3})^{4}=(h_{1,1})^{2}=(h_{1,2})^{2}=0$. Since $c(%
\mathcal{O}(\ell,0,0)\oplus\mathcal{O}(1,1,0))=(1+\ell
h_{3})(1+h_{3}+h_{1,1})$ and 
\begin{equation*}
c(T_{Y})|_{X}=c(T_{X})\cdot c(\mathcal{O}(\ell,0,0)\oplus\mathcal{O}%
(1,1,0)|_{X})
\end{equation*}
it follows that 
\begin{equation*}
c_{1}(T_{X})=(3-\ell)h_{3}+h_{1,1}+2h_{1,2},\ \ \ \ c_{2}(T_{X})=(4-\ell
)h_{3}h_{1,1}+(6-2\ell)h_{3}h_{1,2}+(3-3\ell+%
\ell^{2})h_{3}^{2}+2h_{1,1}h_{1,2}.
\end{equation*}
We calculate the Segre classes $s_{i}(\mathcal{N}_{{\widetilde{\mathcal{F}}}%
}^{\ast})$ for $i=1,\ldots,r$. Since in our example $r=2$, then it is enough
to calculate $s_{i}(\mathcal{N}_{{\widetilde{\mathcal{F}}}}^{\ast})$, $i=1,2$%
. The foliation $\mathcal{F}$ is the restriction of $\rho:Y=\mathbb{P}%
^{3}\times\mathbb{P}^{1}\times\mathbb{P}^{1}\rightarrow\mathbb{P}^{1}\times%
\mathbb{P}^{1}$ to $X$, then the normal bundle of $\mathcal{F}$ is 
\begin{equation*}
N_{\mathcal{F}}=\rho^{\ast}(T_{\mathbb{P}^{1}}\oplus T_{\mathbb{P}%
^{1}})|_{X}=(\mathcal{O}(0,2,0)\oplus\mathcal{O}(0,0,2))|_{X}.
\end{equation*}
Thus $N_{\mathcal{F}}^{\ast}=(\mathcal{O}(0,-2,0)\oplus\mathcal{O}%
(0,0,-2))|_{X}$. Since $(h_{1,1})^{2}=(h_{1,2})^{2}=0$ we get 
\begin{equation*}
s_{1}(N_{\mathcal{F}}^{\ast})=2(h_{1,1}+h_{1,2}),\ \ \ s_{2}(N_{\mathcal{F}%
}^{\ast})=4h_{1,1}h_{1,2}.
\end{equation*}
Observe that 
\begin{align*}
c_{1}(T_{X})\cap s_{1}(\mathcal{N}_{\mathcal{F}}^{\ast}) & =((3-\ell
)h_{3}+h_{1,1}+2h_{1,2})\cdot(2(h_{1,1}+h_{1,2})) \\
& =(6-2\ell)h_{3}h_{1,1}+(6-2\ell)h_{3}h_{1,2}+6h_{1,1}h_{1,2}.
\end{align*}
Thus%
\begin{align*}
& c_{2}(T_{X})-c_{1}(T_{X})\cap s_{1}(\mathcal{N}_{\mathcal{F}%
}^{\ast})+s_{2}(\mathcal{N}_{\mathcal{F}}^{\ast})= \\
&
=(4-\ell)h_{3}h_{1,1}+(6-2\ell)h_{3}h_{1,2}+(3-3\ell+%
\ell^{2})h_{3}^{2}+2h_{1,1}h_{1,2}- \\
&
-((6-2\ell)h_{3}h_{1,1}+(6-2%
\ell)h_{3}h_{1,2}+6h_{1,1}h_{1,2})+4h_{1,1}h_{1,2} \\
& =(\ell-2)h_{3}h_{1,1}+(3-3\ell+\ell^{2})h_{3}^{2}.
\end{align*}
Moreover, we have 
\begin{equation*}
\lbrack
X]=[V(x_{0}^{\ell}+x_{1}^{\ell}+x_{2}^{\ell}+x_{3}^{\ell})]\cap\lbrack
V(x_{0}y_{0}+x_{1}y_{1})]=\ell h_{3}(h_{3}+h_{1,1})=\ell h_{3}^{2}+\ell
h_{3}h_{1,1}.
\end{equation*}
Thus 
\begin{equation*}
\{\left. c_{2}(T_{X})-c_{1}(T_{X})\cap s_{1}(\mathcal{N}_{\mathcal{F}}^{\ast
})+s_{2}(\mathcal{N}_{\mathcal{F}}^{\ast})\right\} \cap\lbrack X]=[(\ell
-2)h_{3}h_{1,1}+(3-3\ell+\ell^{2})h_{3}^{2}]\cdot\lbrack\ell h_{3}^{2}+\ell
h_{3}h_{1,1}].
\end{equation*}
Finally, we obtain 
\begin{align*}
\{\left. c_{2}(T_{X})-c_{1}(T_{X})\cap s_{1}(\mathcal{N}_{\mathcal{F}}^{\ast
})+s_{2}(\mathcal{N}_{\mathcal{F}}^{\ast})\right\} \cap\lbrack X] &
=\ell(\ell-2+3-3\ell+\ell^{2})]h_{3}^{3}h_{1,1} \\
& =\ell(\ell-1)^{2}h_{3}^{3}h_{1,1}=[S].
\end{align*}

\subsection{Non-integrable example}

Let $X$ be a complex-projective manifold of dimension $\dim(X)=2n+1$. A
contact structure on $X$ is a regular distribution ${\mathcal{F}}$ induced
by a twisted $1$-form 
\begin{equation*}
\omega\in H^{0}(X,\Omega_{X}^{1}\otimes L),
\end{equation*}
such that $\omega\wedge(d\omega
)^{n}\neq0$ and  $L$ is a holomorphic line bundle.  Suppose that the second Betti number of $X$ is $b_{2}(X)=1$ and
that $X$ is not isomorphic to the projective space $\mathbb{P}^{2n+1}$. Then
it follows from \cite{Keb} that there exists a compact irreducible component 
$H\subset {\text{\rm RatCurves} }^{n}(X)$ of the space of rational curves on $X$
such that the intersection of $L$ with the curves associated with $H$ is $1.$
Moreover, if $C\subset X$ is a generic element of $H$, then $C$ is smooth,
tangent to ${\mathcal{F}}$, and 
\begin{equation*}
TX|_{C}=\mathcal{O}_{C}(2)\oplus\mathcal{O}_{C}(1)^{n-1}\oplus\mathcal{O}%
_{C}^{n+1},
\end{equation*}%
\begin{equation*}
T_{{\mathcal{F}}}|_{C}=\mathcal{O}_{C}(2)\oplus\mathcal{O}_{C}(1)^{n-1}\oplus%
\mathcal{O}_{C}^{n-1}\oplus\mathcal{O}_{C}(-1).
\end{equation*}
See \cite[Fact 2.2 and Fact 2.3]{KebII}. In particular, we obtain that $%
N_{C|X}=\mathcal{O}_{C}(1)^{n-1}\oplus\mathcal{O}_{C}^{n+1}$, since $T_{C}=%
\mathcal{O}_{C}(2)$. Then 
\begin{equation*}
\det(T_{{\mathcal{F}}})|_{C}-\det(N_{C|X})=\mathcal{O}_{C}(1)
\end{equation*}
is ample. Examples of such manifolds are given by homogeneous Fano contact
manifolds, cf. \cite{Bea}. This example satisfies the conditions of
Corollary \ref{Tangency}.

\end{document}